\documentclass[a4paper,12pt]{amsart}

\usepackage[T1]{fontenc}
\usepackage[utf8]{inputenc}
\usepackage[english]{babel}
\usepackage[colorlinks,plainpages]{hyperref}
\usepackage{csquotes}
\usepackage{amssymb,amsfonts,amsxtra,amsmath}
\usepackage[all,cmtip]{xy}
\usepackage{dsfont}
\renewcommand{\mathbb}{\mathds}
\usepackage{mathrsfs}
\DeclareMathAlphabet{\mathsc}{U}{rsfs}{m}{n}
\usepackage{stmaryrd}
\usepackage[colorlinks,plainpages]{hyperref}
\usepackage{paralist}
\usepackage{url}

\usepackage[backend=biber,style=alphabetic,doi=true,url=false]{biblatex}
\bibliography{ilmis}

\theoremstyle{plain}
\newtheorem{thm}{Theorem}[section]
\newtheorem{cor}[thm]{Corollary} 
\newtheorem{prp}[thm]{Proposition} 
\newtheorem{lem}[thm]{Lemma} 
\theoremstyle{definition}
\newtheorem{dfn}[thm]{Definition}

\theoremstyle{remark} 
\newtheorem{rmk}[thm]{Remark}

\newtheorem{ntn}[thm]{Notation}
\newtheorem{exa}[thm]{Example}

\newcommand{\ol}[1]{{\overline{#1}}}
\newcommand{\ul}[1]{{\underline{#1}}}
\newcommand{\oul}[1]{{\overline{\underline{#1}}}}

\newcommand{\abs}[1]{{\left\vert#1\right\vert}}

\newcommand{\fs}[1]{{\llbracket#1\rrbracket}}
\newcommand{\ideal}[1]{{\left\langle#1\right\rangle}}
\newcommand{\set}[1]{{\left\{#1\right\}}}

\newcommand{\onto}{\twoheadrightarrow}

\newcommand{\xmid}{\;\middle|\;}

\newcommand{\CC}{\mathbb{C}}
\newcommand{\KK}{\mathbb{K}}

\newcommand{\MM}{\mathbb{M}}
\newcommand{\NN}{\mathbb{N}}
\newcommand{\ZZ}{\mathbb{Z}}

\newcommand{\mm}{\mathfrak{m}}
\newcommand{\pp}{\mathfrak{p}}

\DeclareMathOperator{\ann}{ann}
\DeclareMathOperator{\Ass}{Ass}
\DeclareMathOperator{\End}{End}
\DeclareMathOperator{\grade}{grade}
\DeclareMathOperator{\gr}{gr}
\DeclareMathOperator{\Hom}{Hom}
\DeclareMathOperator{\intHom}{{^*\!}\Hom}
\DeclareMathOperator{\ord}{ord}
\DeclareMathOperator{\soc}{soc}
\DeclareMathOperator{\socdeg}{socdeg}
\DeclareMathOperator{\type}{type}

\begin{document}

\title[Inverse limits of Macaulay's inverse systems]{Inverse limits of\\Macaulay's inverse systems}

\author[M.~Schulze]{Mathias Schulze}
\address{M.~Schulze\\
Department of Mathematics\\
TU Kaiserslautern\\
67663 Kaiserslautern\\
Germany}
\email{\href{mailto:mschulze@mathematik.uni-kl.de}{mschulze@mathematik.uni-kl.de}}

\author[L.~Tozzo]{Laura Tozzo}
\address{L.~Tozzo\\
Department of Mathematics\\
TU Kaiserslautern\\
67663 Kaiserslautern\\
Germany}
\email{\href{mailto:tozzo@mathematik.uni-kl.de}{tozzo@mathematik.uni-kl.de}}


\subjclass[2010]{Primary 16D90; Secondary 13A30, 13H10}





\keywords{Macaulay inverse system, Matlis duality, Rees isomorphism}

\begin{abstract}
Generalizing a result of Masuti and the second author, we describe inverse limits of Macaulay's inverse systems for Cohen--Macaulay factor algebras of formal power series or polynomial rings over an infinite field.
On the way we find a strictness result for filtrations defined by regular sequences. 
It generalizes both a lemma of Uli Walther and the Rees isomorphism.
\end{abstract}

\maketitle


\numberwithin{equation}{section}

\section*{Introduction}

Let $\KK$ be a field, and let $P$ be either the (standard graded) polynomial ring $\KK[x_1,\dots,x_n]$ or the formal power series ring $\KK\fs{x_1,\dots,x_n}$ (with trivial grading).
The injective hull $E$ of $\KK$ over $P_0$ defines a duality $-^\vee=\intHom_{P_0}(-,E)$ between Artinian and finitely generated (graded) $P$-modules.
In particular, this yields an antiisomorphism $-^\perp$ of the lattices of (graded) ideals $I\lhd P$ and (graded) $P$-submodules $W$ of $D=P^\vee$ (see \eqref{39}).
The ideals $I$ for which $P/I$ is Artinian correspond to finitely generated submodules $W=I^\perp$.
In this case and for the polynomial ring the correspondence was proved by Macaulay at the beginning of the 20th century.
The dual $I^{\perp}$ is known as the inverse system of $I$ (see \cite{Mac16}). 

Around 1960 Macaulay's correspondence turned out to be a special case of  Matlis duality (see \cite{Mat58,Gab60}).
Later it was rediscovered and further developed (see for instance \cite{Ems78,Iar94,Ge96,GS98,Kle07,CI12}).
Recent applications concern 
the $n$-factorial conjecture (see \cite{Ha94}), 
Waring's problem (see \cite{Ge96}), 
the geometry of the punctual Hilbert scheme of Gorenstein schemes (see \cite{IK99}),
the analytic classification of Artinian Gorenstein rings (see \cite{ER12}),  the cactus rank (see \cite{RS13}), and 
the Kaplansky--Serre problem (see \cite{RS14}).

Elias and Rossi (see \cite{ER17}) described the first generalization of Macaulay's inverse system in the positive-dimensional case. 
Their result which applies to Gorenstein algebras was extended by Masuti and the second author (see \cite{MT18}) to the case of level algebras.
We give a more conceptual description of their construction in terms of inverse limits.
We show how to drop the level-hypothesis by identifying the various socles in the inverse system (see Corollary~\ref{14}).
This fact is deduced from a general strictness result for filtrations defined by regular sequences, which generalizes both a lemma of Walther (see \cite[Lem.~6.5]{Wal17}) and the Rees isomorphism.
The full generality of this result is not needed for our application.

Our main result (see Theorem~\ref{25}) gives an explicit description of inverse limits of Macaulay's inverse systems obtained by dividing out powers of a linear regular sequence.
It applies to (graded) Cohen--Macaulay factor algebras of formal power series (or polynomial) rings over an infinite field (see Lemma~\ref{24} for a more intrinsic description of these types of algebras).

\section{Strict filtrations by regular sequences}\label{33}

We underline vectors and denote (component-wise) residue classes by an overline.
We apply maps component-wise to vectors.
All rings considered are commutative unitary.
We use the ideal symbol $\lhd$.
By an \emph{$R$-sequence} we mean a regular sequence in $R$.


Let $R$ be a ring.
Any ideal $I\lhd R$ defines an \emph{exhaustive decreasing filtration}
\[
R=I^0\supset I\supset I^2\supset I^3\supset\cdots
\]
on $R$ denoted by $I^\bullet$.
It is called \emph{separated} if $\bigcap_{k\in\NN}I^k=0$.
The $I$-order of $p\in R$ is 
\[
\ord_I(p)=\max\set{k\in\NN\xmid p\in I^k}\in\NN\cup\set{\infty}.
\]
We abbreviate $R_I:=R/I$.
The \emph{associated $I$-graded ring} 
\[
\gr_IR=\bigoplus_{l\in\NN}I^l/I^{l+1}
\]
is a homogeneous graded $R_I$-algebra.
There is an \emph{$I$-symbol map}
\[
\sigma_I\colon R\setminus\bigcap_{k\in\NN}I^k\to\gr_IR,\quad 
p\mapsto\ol p\in\gr_I^{\ord_I(p)}R=I^{\ord_I(p)}/I^{\ord_I(p)+1}.
\]
Any ideal $J\lhd R$ induces a filtration $\gr_IJ^\bullet$ on $\gr_IR$, where
\begin{align}\label{16}
\gr_I^lJ^k&=(J^k\cap I^l+I^{l+1})/I^{l+1}\subset\gr_I^lR,\\
\gr_IJ^k&=\bigoplus_{l\in\NN}\gr_I^lJ^k\subset\gr_IR.\nonumber
\end{align}
We refer to any filtration induced by $J$ as a \emph{$J$-filtration}.
If $J$ is generated by $\ul f=f_1,\dots,f_r\in R$, then we use $\ul f$ as a shortcut for the $J$-filtration.


\begin{rmk}\label{18}
Let $\sigma_\ul g(p),\sigma_\ul g(q)\in\gr_\ul g R$.
Then
\[
\ord_\ul g(pq)\ge\ord_\ul g p+\ord_\ul g q,
\]
with equality equivalent to $pq\not\in I^{\ord_\ul g p+\ord_\ul g q+1}$.
It follows that
\[
\sigma_\ul g(p)\sigma_\ul g(q)=
\begin{cases}
\sigma_\ul g(pq) & \text{ if }\ord_\ul g(pq)=\ord_\ul g p+\ord_\ul g q,\\
0 & \text{ otherwise}.
\end{cases}
\]
\end{rmk}


Let $\ul g=(g_1,\dots,g_s)\in R^s$ and denote by $\ul Y=(Y_1,\dots,Y_s)$ corresponding indeterminates of degree $1$.

\begin{thm}[Rees]\label{30}
The Rees map of graded $R_\ul g$-algebras
\begin{equation}\label{15}
\xymatrix@R=0em{
\llap{$\varphi_\ul g\colon$}R_\ul g[\ul Y]\ar[r] & \gr_\ul g R,\\
Y_i\ar@{|->}[r] & \sigma_\ul g(g_i),
}
\end{equation}
is an isomorphism if $\ul g$ is an $R$-sequence (see \cite[Thm.~1.1.8]{BH93}).\qed
\end{thm}


\begin{rmk}\label{19}
If $P\in R[\ul Y]_l$ such that $\ord_\ul g(P(\ul g))=l$, then, using Remark~\ref{18},
\[
\sigma_\ul g(P(\ul g))=\ol P(\sigma_\ul g(\ul g))=\varphi_\ul g\bigl(\ol P\bigr),
\]
where $P\mapsto\ol P$ under $R[\ul Y]\to R_\ul g[\ul Y]$.
\end{rmk}


\begin{rmk}\label{20}
If \eqref{15} is an isomorphism, then the components of $\sigma_\ul g(\ul g)$ are regular on $\gr_\ul g R$.
With Remark~\ref{18}, it follows that  
\[
\sigma_\ul g(p\ul g^\ul m)=
\sigma_\ul g(p)\sigma_\ul g(\ul g^\ul m)=
\sigma_\ul g(p)\sigma_\ul g(\ul g)^\ul m,
\]
for all $\sigma_\ul g(p)\in\gr_\ul g R$ and $\ul m\in\NN^s$.
By Theorem~\ref{30}, this holds if $\ul g$ is an $R$-sequence.
\end{rmk}


Let $\ul f=(f_1,\dots,f_r)\in R^r$, and set
\[
\ul h=(h_1,\dots,h_t)=\ul f,\ul g\in R^r\times R^s=R^t.
\]
Denote by $\ul X=(X_1,\dots,X_r)$ indeterminates of degree $1$ corresponding to $\ul f$, and set 
\[
\ul Z=(Z_1,\dots,Z_t)=\ul X,\ul Y.
\]
For $i\in\set{1,\dots,u}$, let
\[
0\ne\ul m_i=(\ul k_i,\ul l_i)\in\NN^r\times\NN^s=\NN^t
\]
be the rows of a matrix 
\begin{equation}\label{13}
M=(K L).
\end{equation}
We denote by $\ul h^M,\ul f^K,\ul g^L\in R^u$ the vectors with respective entries $\ul h^{\ul m_i},\ul f^{\ul k_i},\ul g^{\ul l_i}\in R$.
Consider the $R$-linear map 
\[
\xymatrix@R=0em{
\llap{$\ul h^M\colon$}R^u\ar[r] & R,\\
\ul e_i\ar@{|->}[r] & \ul h^{\ul m_i},
}
\]
with image $\ideal{\ul h^M}$.
Assigning degrees $\deg\ul e_i=\abs{\ul l_i}$ to the generators defines a $\ul g$-filtration $\bigoplus_{i=1}^u\ideal{\ul g}^{\bullet-l_i}$ on $R^u$ and turns $\ul h^M$ into a filtered map.
It fits into a commutative diagram of filtered $R$-linear maps
\[
\xymatrix@C=1em{
\bigl(R^u,\bigoplus_{i=1}^u\ideal{\ul g}^{\bullet-l_i}\bigr)\ar[r]^-{\ul h^M}\ar[d]^-{\ul h^M} & \bigl(R,\ideal{\ul g}^\bullet\bigr) \\
\bigl(\ideal{\ul h^M},\sum_{i=1}^u\ideal{\ul g}^{\bullet-l_i}\ul h^{\ul m_i}\bigr)\ar[r] & \bigl(\ideal{\ul h^M},\ideal{\ul g}^\bullet\cap\ideal{\ul h^M}\bigr).\ar[u]
}
\]
The bottom map is the identity of $\ideal{\ul h^M}$, but its source and target carry respectively the image and preimage filtration from the source and target of the (horizontal) map $\ul h^M$.
If it identifies the two filtrations, then $\ul h^M$ is said to be $\ul g$-strict.
The vertical maps are $\ul g$-strict by construction.


The following proposition gives a generalized Rees isomorphism.


\begin{prp}\label{2}
Suppose that both $\ul h=\ul f,\ul g$ and $\ul g,\ul f$ are $R$-sequences, and that the $\oul f$-filtration on $R_\ul g$ is separated and complete.
Then $\ul h^M$ is $\ul g$-strict.
In particular, the Rees map \eqref{15} induces an isomorphism of graded $R_\ul g$-algebras
\[
\xymatrix@R=0em{
R_\ul g[\ul Y]/\ideal{\oul f^K\ul Y^L}\ar[r]^-{\ol{\varphi_\ul g}}_-\cong & \gr_\ul g(R)/\ideal{\sigma_\ul g(\ul h^M)}\ar[r]_-\cong & \gr_{\oul g}\bigl(R/\ideal{\ul h^M}\bigr),\\
& \ol{\sigma_\ul g(x)}\ar@{<->}[r] & \sigma_{\oul g}(\ol x),
}
\]
where $\oul f^K\ul Y^L$ denotes the vector with entries $\oul f^{\ul k_i}\ul Y^{\ul l_i}$.
\end{prp}


The proof of Proposition~\ref{2} relies on the following lemma proved by Uli Walther for $k=1$ (see \cite[Lem.~6.5]{Wal17}). 
He assumes that $R$ is a domain, and that $\ul g,\ul f$ is an $R$-sequence in every order.
However, his proof needs only that $\ul f,\ul g$ is an $R$-sequence.


\begin{lem}\label{1}
Suppose that $\ul g$ and $\ul f,\ul g$ are $R$-sequences.
Let $P\in R[\ul Y]_l$ such that $P(\ul g)\in\ideal{\ul f}^k$.
Then $P(\ul g)=Q(\ul g)$ for some $Q\in\ideal{\ul f}^k[\ul Y]_l$.
In particular, the Rees map \eqref{15} induces a filtered isomorphism
\[
\xymatrix{
\varphi_\ul g\colon\bigl(R_\ul g[\ul Y],\ideal{\oul f}^\bullet[\ul Y]\bigr)\ar[r] & \bigl(\gr_\ul g R,\gr_\ul g\ideal{\ul f}^\bullet\bigr).
}
\]
\end{lem}

\begin{proof}
We proceed by induction on $k$.
The claim is vacuous for $k=0$.
By Walther's lemma (see \cite[Lem.~6.5]{Wal17}), we may assume that $P=\sum_iP_if_i\in\ideal{\ul f}[\ul Y]_l$ with $P_i\in R[\ul Y]_l$, and hence $P(\ul g)=\sum_iP_i(\ul g)f_i\in\ideal{\ul f}^k$.
The following proof of Proposition~\ref{2} relies only on the particular claim, which reduces to the Rees isomorphism in case $r=0$.
Applying Proposition~\ref{2} with $r=0$, $\ul f$ playing the role of $\ul g$ and $\ul m_i=\ul e_i$, it follows that $P_i(\ul g)\in\ideal{\ul f}^{k-1}$.
By induction hypothesis, $P_i(\ul g)=Q_i(\ul g)$ for some $Q_i\in\ideal{\ul f}^{k-1}[\ul Y]_l$.
Then $P(\ul g)=Q(\ul g)$ for $Q=\sum_iQ_if_i\in\ideal{\ul f}^k[\ul Y]_l$.
This proves the first claim.

Suppose now that $\ul g$ is an $R$-sequence.
Then $\varphi_\ul g$ is an isomorphism by Theorem~\ref{30}.
Clearly $\varphi_\ul g\bigl(\ideal{\oul f}^\bullet[\ul Y]\bigr)\subset\gr_\ul g\ideal{\ul f}^\bullet$ (see \eqref{16}).
For the converse inclusion, take $\sigma_\ul g(x)\in\gr_\ul g^l\ideal{\ul f}^k$.
Then $x=P(\ul g)\in\ideal{\ul f}^k$ for some $P\in R[\ul Y]_l$.
By the first claim, we may assume that $P\in\ideal{\ul f}^k[\ul Y]_l$.
Then $R[\ul Y]\to R_\ul g[\ul Y]$ maps $P\mapsto\ol P\in\ideal{\oul f}^k[\ul Y]_l$ with $y=\sigma_\ul g(x)=\sigma_\ul g(P(\ul g))=\varphi_\ul g\bigl(\ol P\bigr)$ by Remark~\ref{19}.
This shows that $\varphi_\ul g\bigl(\ideal{\oul f}^k[\ul Y]_l\bigr)=\gr_\ul g^l\ideal{\ul f}^k$ and the particular claim follows.
\end{proof}


A second ingredient of the proof of Proposition~\ref{2} is the following general relation of strict and graded exactness.


\begin{lem}\label{7}\pushQED{\qed}
Let $A$ be a filtered ring, and let
\[
\xymatrix{
C\colon N'\ar[r]^-{\alpha'} & N\ar[r]^-\alpha & N''
}
\]
be a filtered complex of $A$-modules with associated graded complex $\gr C$.
\begin{enumerate}[(a)]

\item\label{7a} If $C$ is strict exact, then $\gr C$ is exact (see \cite[Lem.~1.(a)]{Sjo73}).

\item\label{7b} If $\gr C$ is exact and the filtration on $N$ is exhaustive, then $\alpha$ is strict (see \cite[Lem.~1.(b)]{Sjo73}).

\item\label{7e} If the filtration on $N'$ is complete and the filtration on $N$ is exhaustive and separated, then $C$ is strict exact if and only if $\gr C$ is exact (see \cite[Lem.~1.(e)]{Sjo73}).\qedhere

\end{enumerate}
\end{lem}


\begin{proof}[Proof of Proposition~\ref{2}]
Set $U=\set{(i,j)\mid 1\le i<j\le u}$ and consider the $R$-linear map
\[
\xymatrix@R=0em{
R^U\ar[r] & R^u,\\
\ul e_{i,j}\ar@{|->}[r] & \ul h^{\ul m_j-\ul\min\set{\ul m_i,\ul m_j}}\ul e_i-\ul h^{\ul m_i-\ul\min\set{\ul m_i,\ul m_j}}\ul e_j,
}
\]
where $\ul\min$ denotes the component-wise minimum.
Assign to the generators bidegrees 
\[
\deg(\ul e_i)=(\abs{\ul l_i},\abs{\ul k_i}),\quad \deg(\ul e_{i,j})=(\abs{\ul l_i}+\abs{\ul l_j},\abs{\ul k_i}+\abs{\ul k_j}).
\]
With component-wise $\ul f$- and $\ul g$-filtrations,
\[
\xymatrix@R=0em{
C\colon R^U\ar[r] & R^u\ar[r]^-{\ul h^M} & R
}
\]
becomes a bifiltered complex of free $R$-modules.
By Lemma~\ref{7}.\eqref{7b}, it suffices to show that the $\ul g$-graded complex $\gr_\ul gC$ is exact.
This can be checked on graded pieces.
By Lemma~\ref{7}.\eqref{7e}, it suffices to show that the $\ul f$-filtration $\gr_\ul g\ideal{\ul f}^\bullet$ is separated and complete on each graded piece of $\gr_\ul gC$, and that the associated graded complex $\gr_\ul f\gr_\ul gC$ is exact.

By Lemma~\ref{1}, the Rees isomorphism $\varphi_\ul g\colon R_\ul g[\ul Y]\to\gr_\ul gR$ identifies the induced $\ul f$-filtration $\gr_\ul g\ideal{\ul f}^\bullet$ on $\gr_\ul gR$ with the $\oul f$-filtration on the coefficient ring $R_\ul g$.
We denote by $\gr_\ul f\varphi_\ul g$ the associated graded isomorphism.
The graded pieces of $\gr_\ul gR$, and hence of $\gr_\ul gC$, are finite direct sums of $R_\ul g$.
By hypothesis, the $\oul f$-filtration is separated and complete on each summand.
It follows that the induced $\ul f$-filtration is separated and complete on each graded piece of $\gr_\ul gC$.

By Theorem~\ref{30} applied to the regular $R_\ul g$-sequence $\oul f$ and Lemma~\ref{1}, there is a bigraded isomorphism of $R_\ul h$-algebras
\begin{equation}\label{5}
\xymatrix@C=3em{
R_\ul h[\ul Z]\cong(R_\ul g)_{\oul f}[\ul X][\ul Y]\ar[r]^-{\varphi_{\oul f}[\ul Y]}_-\cong &
\gr_{\oul f}(R_\ul g)[\ul Y]\ar[r]^-{\gr_\ul f\varphi_\ul g}_-\cong & 
\gr_\ul f\gr_\ul gR.
}
\end{equation}
Let now $\ul m=(\ul k,\ul l)\in\NN^r\times\NN^s=\NN^t$.
Note that $\oul f^\ul k\ne0$ by Remark~\ref{20} applied to the $R_\ul g$-sequence $\oul f$.
By $R_\ul g$-linearity of $\varphi_\ul g$ and Remark~\ref{20},
\begin{equation}\label{4}
\varphi_\ul g\bigl(\oul f^{\ul k}\ul Y^{\ul l}\bigr)=
\oul f^{\ul k}\varphi_\ul g(\ul Y^{\ul l})=
\sigma_\ul g(\ul f^{\ul k})\sigma_\ul g(\ul g^{\ul l})=
\sigma_\ul g(\ul f^{\ul k}\ul g^{\ul l})=
\sigma_\ul g(\ul h^\ul m).
\end{equation}
It follows that isomorphism \eqref{5} maps
\[
\xymatrix{
\ul Z^{\ul m}=\ul X^{\ul k}\ul Y^{\ul l}\ar@{|->}[r] & \sigma_{\oul f}\bigl(\oul f^{\ul k}\bigr)\ul Y^{\ul l}\ar@{|->}[r] & \sigma_\ul f\bigl(\varphi_\ul g\bigl(\oul f^{\ul k}\ul Y^{\ul l}\bigr)\bigr)=\sigma_\ul f(\sigma_\ul g(\ul h^\ul m)).
}
\]
The isomorphism \eqref{5} thus turns $\gr_\ul f\gr_\ul gC$ into the exact complex (see \cite[Lem.~15.1]{Eis95})
\[
\xymatrix@R=0em{
R_\ul g[\ul Z]^U\ar[r] & R_\ul g[\ul Z]^u\ar[r] & R_\ul g[\ul Z],\\
\ul e_{i,j}\ar@{|->}[r] & \ul Z^{\ul m_j-\ul\min\set{\ul m_i,\ul m_j}}\ul e_i-\ul Z^{\ul m_i-\ul\min\set{\ul m_i,\ul m_j}}\ul e_j,\\
& \ul e_i\ar@{|->}[r] & \ul Z^{\ul m_i},
}
\]
which proves the first claim.

With $R/\ideal{\ul h^M}$ equipped with the image $\ul g$-filtration,
\[
\xymatrix@R=0em{
R^u\ar[r]^-{\ul h^M} & R\ar[r] & R/\ideal{\ul h^M}\ar[r] & 0
}
\]
is an exact complex of $\ul g$-strict $R$-linear maps.
Then the corresponding $\ul g$-graded complex is exact by Lemma~\ref{7}.\eqref{7a}, and hence
\[
\gr_\ul g\bigl(R/\ideal{\ul h^M}\bigr)\cong\gr_\ul g(R)/\gr_\ul g\ideal{\ul h^M}\cong\gr_\ul g(R)/\ideal{\sigma(\ul h^M)}.
\]
With \eqref{4}, this proves the particular claim.
\end{proof}


We now specialize to the case where $\ul g=\ul h$ and $R$ is Noetherian \textsuperscript{*}local graded with \textsuperscript{*}maximal ideal $\mm_R$.
Denote the \emph{(homogeneous) socle} of $R$ by
\begin{equation}\label{23}
\soc R=\ann_R\mm_R.
\end{equation}
We assume that $\ul h$ has homogeneous components making $\ideal{\ul h}\lhd R$ a graded ideal.
Then $\ideal{\ul h}\subset\mm_R$ and the filtration $\ideal{\ul h}^\bullet$ is separated by Krull intersection theorem (see \cite[\S3.1, Thm.~1]{Nor53}) and Nakayama lemma (see \cite[Ex.~1.5.24]{BH93}).
With $R$ also $\gr_\ul h^0 R=R_\ul h$ is Noetherian \textsuperscript{*}local graded with \textsuperscript{*}maximal ideal $\mm_{R_\ul h}=\mm_R/\ideal{\ul h}$.
The $R_\ul h$-algebra $\gr_\ul hR$ is now bigraded with unique bigraded maximal ideal
\begin{equation}\label{26}
\mm_{\gr_\ul hR}=\mm_{R_\ul h}+\ideal{\sigma_\ul h(\ul h)}.
\end{equation}
For any bigraded algebra $S$ with unique bigraded maximal ideal $\mm_S$, we define the \emph{(bihomogeneous) socle} as in \eqref{23}.


The rows of the matrix $M$ from \eqref{13} generate a monoid ideal
\[
\MM=\ideal{\ul m_1,\dots,\ul m_u}\subset\NN^t.
\]
By Dickson's lemma (see \cite{Dic13}), every monoid ideal is finitely generated.


\begin{dfn}\label{17}
Let $\MM\subset\NN^t$ be a monoid ideal.
By the \emph{socle} of $\MM$ we mean the subset
\[
\soc\MM=\set{\ul n\in\NN^t\setminus\MM\xmid\ul n+(\NN^t\setminus\set{0})\subset\MM}\subset\NN^t.
\]
For $d\in\NN$ we write $\soc_d\MM=\set{\ul m\in\soc\MM\mid\abs{\ul m}=d}$.
\end{dfn}


By definition,
\begin{equation}\label{12}
\ann_{R[\ul Z]/\ideal{\ul Z^M}}\bigl(\oul Z\bigr)=
\bigoplus_{\ul m\in\soc\MM}R\oul Z^\ul m\cong R^{\soc\MM}.
\end{equation}
This is a Noetherian $R$-module if $R$ is a Noetherian ring.
In particular, $\soc\MM$ is a finite set.
Using \eqref{12} and $\mm_{R[\ul Z]/\ideal{\ul Z^M}}=\mm_R+\ideal{\oul Z}$, we find
\begin{equation}\label{28}
\soc\bigl(R[\ul Z]/\ideal{\ul Z^M}\bigr)=
\bigoplus_{\ul m\in\soc\MM}\soc(R)\oul Z^\ul m\cong\soc(R)^{\soc\MM}.
\end{equation}


\begin{cor}\label{14}
Suppose that $R$ is a Noetherian \textsuperscript{*}local graded ring, and that $\ul h$ a (component-wise) homogeneous $R$-sequence. 
Then the symbol map (extended by zero)
\[
\xymatrix@R=0em{
R/\ideal{\ul h^M}\ar[r] & \gr_\oul h\bigl(R/\ideal{\ul h^M}\bigr),\\
\ol x\ar@{|->}[r] & \sigma_\oul h(\ol x),
}
\]
identifies socles.
\end{cor}

\begin{proof}

Proposition~\ref{2} yields an isomorphism of free $R_\ul h$-modules
\begin{equation}\label{8}
\xymatrix@R1em{
\ann_{\gr_\ul h(R)/\ideal{\sigma_\ul h(\ul h^M)}}\bigl(\ol{\sigma_\ul h(\ul h)}\bigr)\ar[r]_-\cong & 
\ann_{\gr_\oul h\left(R/\ideal{\ul h^M}\right)}\bigl(\sigma_\oul h\bigl(\oul h\bigr)\bigr)\\
\bigoplus_{\ul m\in\soc\MM}R_\ul h\ol{\sigma_\ul h(\ul h^\ul m)}\ar@{=}[u] &
\bigoplus_{\ul m\in\soc\MM}R_\ul h\sigma_\oul h\bigl(\oul h^\ul m\bigr),\ar@{=}[u]\\
\ol{\sigma_\ul h(x)}\ar@{|->}[r] & \sigma_\oul h(\ol x).
}
\end{equation}
Since $\sigma_\oul h\bigl(\oul h\bigr)\in\mm_{\gr_\oul h\left(R/\ideal{\ul h^M}\right)}$ by \eqref{26}, it follows that
\[
\soc\bigl(\gr_\oul h\bigl(R/\ideal{\ul h^M}\bigr)\bigr)\subset\bigoplus_{\ul m\in\soc\MM}R_\ul h\sigma_\oul h\bigl(\oul h^\ul m\bigr).
\]
Moreover, the $R_\ul h$-linear surjection
\[
\xymatrix{
R/\ideal{\ul h^M}\supset\sum_{\ul m\in\soc\MM}R_{\ul h}\oul h^{\ul m}\ar@{->>}[r] & \bigoplus_{\ul m\in\soc\MM}R_{\ul h}\sigma_{\oul h}\bigl(\oul h^\ul m\bigr)\cong R_\ul h^{\soc\MM}
}
\]
onto the free $R_\ul h$-module must be an isomorphism, and hence
\[
R/\ideal{\ul h^M}\supset\bigoplus_{\ul m\in\soc\MM}R_{\ul h}\oul h^{\ul m}
\cong\bigoplus_{\ul m\in\soc\MM}R_{\ul h}\sigma_{\oul h}\bigl(\oul h^\ul m\bigr)
\subset\gr_\oul h\bigl(R/\ideal{\ul h^M}\bigr)
\]
is an isomorphism of free $R_\ul h$-modules.
The action of the respective graded and bigraded maximal ideal on these modules reduces to that of $\mm_{R_\ul h}$.
Therefore, it remains to show that
\[
\soc\bigl(R/\ideal{\ul h^M}\bigr)\subset\bigoplus_{\ul m\in\soc\MM}R_\ul h\oul h^\ul m.
\]
To this end, let $0\ne\ol x\in\soc\bigl(R/\ideal{\ul h^M}\bigr)$ of $\oul h$-order $d=\ord_\oul h(\ol x)$.
In particular, $x\ideal{\ul h}\subset\ideal{\ul h^M}$ since $\ul h\in\mm_R$.
By Remark~\ref{20} and Proposition~\ref{2}, taking symbols yields
\[
\sigma_\ul h(x)\sigma_\ul h(\ul h)=\sigma_\ul h(x\ul h)\in\gr_\ul h\ideal{\ul h^M}=\ideal{\sigma_\ul h(\ul h)^M}.
\]
Then, by \eqref{8}, $\sigma_\oul h(\ol x)\in\ann_{\gr_\oul h(R/\ideal{\ul h^M})}\bigl(\sigma_\oul h\bigl(\oul h\bigr)\bigr)$ can be written as
\[
\sigma_\oul h(\ol x)=\sum_{\ul m\in\soc_d\MM}\ol x_\ul m\sigma_\oul h\bigl(\oul h^\ul m\bigr)\in\gr_\oul h\bigl(R/\ideal{\ul h^M}\bigr),
\]
where $\ol x_\ul m\in R_\ul h$.
With
$x'=x-\sum_{\ul m\in\soc_d\MM}x_\ul m\ul h^\ul m$,
this means that
\[
\ol{x'}=\ol x-\sum_{\ul m\in\soc_d\MM}\ol x_\ul m\oul h^\ul m\in\ideal{\oul h}^{d+1}\lhd R/\ideal{\ul h^M},
\]
and hence $\ord_\oul h\bigl(\ol{x'}\bigr)>d=\ord_\oul h(\ol x)$ if $\ol{x'}\ne0$.
By \eqref{12}, $\ul Z^\ul m\ideal{\ul Z}\in\ideal{\ul Z^M}\lhd R[\ul Z]$.
Substituting $\ul Z=\ul h$ gives $\ul h^\ul m\ideal{\ul h}\subset\ideal{\ul h^M}$, and hence $x'\ideal{\ul h}\subset\ideal{\ul h^M}$.
Since $\soc\MM$ is finite, iterating yields
\[
\ol x=\sum_{\ul m\in\soc\MM}\ol x_\ul m\oul h^\ul m.
\]
The remaining inclusion follows.
\end{proof}

\section{Macaulay's inverse system}\label{40}

Let $P$ be a Noetherian \textsuperscript{*}local graded ring with \textsuperscript{*}maximal ideal $\mm_P$.
Then $P_0$ is Noetherian local with maximal ideal $\mm_{P_0}=(\mm_P)_0$.
We assume that $P$ is \emph{\textsuperscript{*}complete}, which means that $P_0$ is complete.
A $P$-module $M$ is called \emph{\textsuperscript{*}Artinian} if every descending chain of graded $P$-submodules is stationary.
Using Nakayama lemma (see \cite[Ex.~1.5.24]{BH93}), we define its \emph{socle degree} to be the nilpotency index
\[
\socdeg M=\inf\set{k\in\ZZ\mid\mm_P^kM=0}-1\in\NN\cup\set{-\infty}.
\]
Denote by $E_{P_0}(P_0/\mm_{P_0})$ the injective hull of the residue field of $P_0$.


\begin{thm}[Graded Matlis duality]
The dualizing functor 
\[
-^\vee=\intHom_{P_0}(-,E_{P_0}(P_0/\mm_{P_0}))
\]
defines an antiequivalence between the categories of \textsuperscript{*}Artinian and finitely generated graded $P$-modules (see \cite[Thm.~3.6.17]{BH93}).\qed
\end{thm}


With $D=P^\vee$, the functor $-^\vee$ induces an antiisomorphism of lattices
\begin{equation}\label{39}
\xymatrix@R=0em{
\set{I\lhd P\text{ graded ideal}}\ar@{<->}[r]^-\perp & \set{W\subset D\text{ graded $P$-submodule}},\\
I\ar@{|->}[r] & I^\perp\rlap{$=(P/I)^\vee$,}\\
\llap{$(D/W)^\vee=$}W^\perp & W,\ar@{|->}[l]
}
\end{equation}
where $P/I$ is \textsuperscript{*}Artinian if and only if $W$ is finitely generated.


Let $\KK$ be field, and let $\ul x=x_1,\dots,x_n$ be indeterminates, where $n\in\NN\setminus\set0$.
Denote by $P$ either the (standard graded) polynomial ring $\KK[\ul x]$ or the formal power series ring $\KK\fs{\ul x}$, both with $\mm_P=\ideal{\ul x}$.
In both cases, $D$ identifies as a $\KK$-vector space with a polynomial ring $\KK[\ul X]$ in indeterminates $\ul X=X_1\dots,X_n$ with $P$-module structure given by (see \cite[Thm.~2.3.2]{Eli18})
\begin{equation}\label{10}
\ul x^\ul n\cdot\ul X^\ul m=
\begin{cases}
\ul X^{\ul m-\ul n} & \text{if } \ul m\ge\ul n,\\
0 & \text{otherwise}.
\end{cases}
\end{equation}
Note that $(\mm_P^k)^\perp=D_{<k}=\bigoplus_{j=0}^{k-1}D_j$ with $\dim_\KK D_{<k}<\infty$, for all $k\in\NN$.
With \eqref{39}, it follows that
\begin{equation}\label{34}
\max\deg I^\perp=\socdeg(P/I),\quad\dim_\KK I^\perp<\infty,
\end{equation}
if $P/I$ is \textsuperscript{*}Artinian.


In the case where $\KK$ is infinite and $P/I$ is Cohen--Macaulay, the following lemma is the starting point for our explicit description of $I^\perp$.


\begin{lem}\label{24}
Suppose that $R$ is a Noetherian \textsuperscript{*}complete \textsuperscript{*}local homogeneous graded algebra with coefficient field $\KK$.
Then $R\cong P/I$ where $I\lhd\KK\fs{\ul y}[\ul z]=P$ with $P_0=\KK\fs{\ul y}$ and indeterminates $\ul z$ of degree $1$.
Suppose that $P=\KK[\ul x]$ or $P=\KK\fs{\ul x}$, $\KK$ is infinite, and that $R$ is Cohen--Macaulay of dimension $d$.
Then, after a $\KK$-linear change of coordinates, $\ul x=\ul y,\ul z$ and $\ul z$ maps to an $R$-sequence of length $d$.
\end{lem}

\begin{proof}
By hypothesis, $R_0$ is Noetherian and $R_1$ is a finite $R_0$-module (see \cite[Prop.~1.5.4]{BH93}).
Then $R_0\cong\KK\fs{\ul y}/I_0$ by Cohen structure theorem and the first claim follows.
Suppose now that $P=\KK[\ul x]$ or $P=\KK\fs{\ul x}$.
If $\KK$ is infinite and $\grade(\mm_R,R)>0$, then the $\KK$-vector space $\ideal{\oul x}_\KK\cong\mm_R/\mm_R^2$ is not the finite union of proper subspaces $\bigcup_{\pp\in\Ass R}(\pp+\mm_R^2)/\mm_R^2$.
Then some $\KK$-linear combination of $\oul x$ is regular on $R$ and the second claim follows by induction (see \cite[Prop.~1.5.12]{BH93}).
\end{proof}


Let $d\in\set{0,\dots,n}$ and partition $\ul x=\ul y,\ul z$ into sets of indeterminates $\ul y=y_1,\dots,y_{n-d}$ and $\ul z=z_1,\dots,z_d$.
Partition $\ul X=\ul Y,\ul Z$ correspondingly into sets of indeterminates $\ul Y=Y_1,\dots,Y_{n-d}$ and $\ul Z=Z_1,\dots,Z_d$.
The indeterminates $\ul X,\ul Y,\ul Z$ are not related to the ones denoted by the same symbols in \S\ref{33}.
Consider the inverse system over $\NN^d$ defined by 
\[
\ul n\mapsto D,\quad\ul n\le\ul m\mapsto\ul z^{\ul m-\ul n}\in\End_P(D)
\]
with limit $\varprojlim D=\KK[\ul Y]\fs{\ul Z}$.


\begin{ntn}
Consider the $P$-submodules 
\[
V^{j,k}_{\ul m}=\ideal{\ul X^{\ul k}\xmid|\ul k|\le |\ul m|+k,\ \ul k=(\ul l,\ul n),\ n_j<m_j-1}_P\subset D
\]
where $j\in\set{1,\dots,d}$, $k\in\NN$ and $\ul m\in\NN^d$.
\end{ntn}


\begin{rmk}\label{32}
By definition, $V^{j,k}_{\ul m}$ is an intersection of $P$-modules
\[
V^{j,k}_{\ul m}=
\ideal{\ul X^{\ul k}\xmid|\ul k|\le |\ul m|+k}_P\cap
\ideal{\ul Y^\ul l\ul Z^\ul n\xmid n_j<m_j-1}_P
\]
and applying the lattice antiisomorphism~\eqref{39} yields
\begin{align*}
(V^{j,k}_{\ul m})^\perp
&=\ideal{\ul X^{\ul k}\xmid|\ul k|\le |\ul m|+k}_P^\perp
+\ideal{\ul Y^\ul l\ul Z^\ul n\xmid n_j<m_j-1}_P^\perp\\
&=\ideal{\ul x}^{\abs{\ul m}+k+1}+\ideal{z_j^{m_j-1}}.
\end{align*}
\end{rmk}


\begin{dfn}\label{27}
Let $d\in\set{0,\dots,n}$, and let $H\subset\varprojlim D$ be a finite $\KK$-vector subspace.
Denote by $H_\ul m$ its image in the copy of $D$ assigned to $\ul m\in\NN^d$ and consider the $P$-submodule 
\begin{equation}\label{11}
W_{\ul m}=\ideal{H_{\ul m}}_P\subset D.
\end{equation}
We call $H$ a \emph{limit inverse system} of \emph{dimension} $\dim H=d$, \emph{type} $\type H=r\in\NN\setminus\set0$ and \emph{socle degree} $\socdeg H=s\in\NN$ if
\begin{enumerate}[(a)]

\item\label{27b} $\dim_\KK H=r$,

\item\label{27a} $\min\{\ul m\in\NN^d\mid H_{\ul m}\ne 0\}=\ul 1$,

\item\label{27d} $\max\deg H_{\ul m}=|\ul m|+s-d$ and

\item\label{27c} $W_\ul m\cap V^{j,s-d}_\ul m\subset W_{\ul m-\ul e_j}$, 

\end{enumerate}
for all $\ul m\in\NN^d$ and $j\in\set{1,\dots,d}$.
We consider $H\simeq H'$ as equivalent if $W_\ul m=W'_\ul m$, for all $\ul m\in\NN^d$.
\end{dfn}


\begin{rmk}\label{29}\
\begin{asparaenum}[(a)]

\item\label{29a} Condition~\ref{27}.\eqref{27a} implies that for all $\ul m\in\NN^d$ and $i\in\set{1,\dots,d}$, $z_i^{m_i}\cdot H_{\ul m}=0$, and hence $\max\deg_{Z_i} H_{\ul m}=m_i-1$.
In particular, $\max\deg_{\ul Z}H_{\ul m}=|\ul m|-d$.

\item\label{29b} Condition~\ref{27}.\eqref{27d} can be substituted by $\max\deg H_{\ul 1}=s$ and $\max\deg H_{\ul m}\le|\ul m|+s-d$, for all $\ul m\in\NN^d$.

\end{asparaenum}
\end{rmk}


For any $I\lhd P$ and $\ul m\in\NN^d$, we set
\begin{equation}\label{9}
I_{\ul m}=I+\ideal{z_1^{m_1},\dots,z_d^{m_d}}\lhd P,\quad R_\ul m=P/I_\ul m.
\end{equation}


\begin{lem}\label{41}
Any $I\lhd P$ can be recovered from \eqref{9} as
\[
I=\bigcap_{\ul n\in\NN^d} I_\ul n.
\]
\end{lem}

\begin{proof}
This is a consequence of Krull intersection theorem.
\end{proof}


\begin{thm}\label{25}
Let $d\in\set{0,\dots,n}$, $r\in\NN\setminus\set0$ and $s\in\NN$.
Then there is a bijection between

\begin{enumerate}[(a)]

\item\label{25i} the set of (graded) ideals $I\lhd P$ such that $R=P/I$ is Cohen--Macaulay, $\dim R=d$, $\ul z=z_1,\dots,z_d$ maps to an $R$-sequence, $\type R=r$ and $\socdeg R_\ul 1=s$ and

\item\label{25h} the set of limit inverse systems $H\subset\varprojlim D$ with $\dim H=d$, $\type H=r$ and $\socdeg H=s$ modulo equivalence.

\end{enumerate}

The map from \eqref{25i} to \eqref{25h} is defined by setting (see \eqref{9})
\begin{equation}\label{25b}
W_\ul m=I_\ul m^\perp=R_\ul m^\vee\subset D
\end{equation}
and taking $H\subset\varprojlim_{\ul n\in\NN^d}W_\ul n$ the image of a $\KK$-linear section of the canonical surjection
\begin{equation}\label{25l}
\xymatrix{
\varprojlim D\supset
\varprojlim_{\ul n\in\NN^d}W_\ul n\ar@{->>}[r] & \varprojlim_{\ul n\in\NN^d}(W_\ul n\otimes\KK)\cong W_\ul 1\otimes\KK\cong\soc(R_\ul 1)^\vee,
}
\end{equation}
where the inverse systems are defined by $\ul n\le\ul m\mapsto\ul z^{\ul m-\ul n}$.
The map from \eqref{25h} to \eqref{25i} is defined by setting (see \eqref{11})
\begin{equation}\label{25a}
I=\bigcap_{\ul n\in\NN^d}W_{\ul n}^\perp.
\end{equation}
\end{thm}


\begin{lem}\label{22}\
\begin{enumerate}[(a)]

\item\label{22a} Let $I$ be in the set~\ref{25}.\eqref{25i} and $W_\ul m$ as in \eqref{25b}.
Then $R_\ul m$ is Artinian, and hence $\dim_\KK W_\ul m<\infty$.

\item\label{22b} Let $H$ in the set~\ref{25}.\eqref{25h} and $W_\ul m$ as in \eqref{11}.
Then
\[
\max\deg W_\ul m=\abs{\ul m}+s-d,
\]
and hence $\dim_\KK W_\ul m<\infty$.

\end{enumerate}
\end{lem}

\begin{proof}\
\begin{asparaenum}[(a)]

\item Since $\ideal{\ul z}=\sqrt{\ideal{z_1^{m_1},\dots,z_d^{m_d}}}$ and $\ul z$ maps to an $R$-sequence of length $d=\dim R$, $\dim R_\ul m=\dim R_\ul 1=0$.
Then $R_\ul m$ is Artinian by Hopkins theorem, and hence $\dim_\KK W_\ul m<\infty$ by \eqref{34}.

\item This follows from Definition~\ref{27}.\eqref{27d} and \eqref{10}.\qedhere

\end{asparaenum}
\end{proof}


\begin{lem}\label{21}
Let $I$ be in the set~\ref{25}.\eqref{25i} and $W_\ul m$ as in \eqref{25b}.

\begin{enumerate}[(a)]

\item\label{21a} There is a canonical surjection \eqref{25l}.

\end{enumerate}
Let $H\subset\varprojlim_{\ul n\in\NN^d}W_\ul n$ be the image of a $\KK$-linear section of the surjection \eqref{25l}.

\begin{enumerate}[(a)]\setcounter{enumi}{1}

\item\label{21c} The $P$-module $W_\ul m$ is minimally generated by $H_\ul m$, for all $\ul m\in\NN^d$.
In particular, \eqref{11} holds true.

\item\label{21b} The $\KK$-vector space $H$ is in the set~\ref{25}.\eqref{25h}.

\end{enumerate}
\end{lem}

\begin{proof}
In the following, $\ul n,\ul m\in\NN^d$ with $\ul n\le \ul m$.
\begin{asparaenum}[(a)]

\item Consider the surjection of direct systems represented by 
\begin{equation}\label{38}
\xymatrix{
R\ar@{->>}[d]\ar[r]^-{{\ul z}^{\ul m-\ul n}} & R\ar@{->>}[d]\\
R_\ul n\ar[r]^-{{\ul z}^{\ul m-\ul n}} & R_\ul m.
}
\end{equation}
Applying $-^\vee$ yields an inclusion of inverse systems represented by
\[
\xymatrix{
D & D\ar[l]_-{{\ul z}^{\ul m-\ul n}}\\
W_\ul n\ar@{^(->}[u] & W_\ul m.\ar[l]_-{{\ul z}^{\ul m-\ul n}}\ar@{^(->}[u]
}
\]
Left-exactness of the inverse limit then yields the inclusion in \eqref{25l}.

We now apply \S\ref{33} with $\ul g=\ul h=\ul z$ and $M$ the matrix with diagonal $\ul m\in\NN^d$.
Then $\soc\MM=\set{\ul m-\ul1}$ in Definition~\ref{17}.

By Proposition~\ref{2}, Corollary~\ref{14} and \eqref{28}, the bottom map in \eqref{38} identifies (homogeneous) socles.
This yields an inclusion of direct systems represented by
\begin{equation}\label{37}
\xymatrix{
R_\ul n\ar[r]^-{{\ul z}^{\ul m-\ul n}} & R_\ul m\\
\soc R_\ul n\ar@{^(->}[u]\ar[r]_-\cong & \soc R_\ul m.\ar@{^(->}[u]
}
\end{equation}
By Lemma~\ref{22}.\eqref{22a}, $R_\ul m$ is Artinian, and hence (see \cite[Prop.~2.4.3]{Eli18})
\begin{equation}\label{36}
\soc(R_{\ul m})^\vee\cong I_{\ul m}^\perp/\mm_P\cdot I_{\ul m}^\perp\cong W_\ul m\otimes\KK.
\end{equation}
With $\ul m=\ul 1$, this is the second isomorphism in \eqref{25l}.

Applying \eqref{36} to the bottom row of \eqref{37}, this yields a trivial inverse system represented by
\[
\xymatrix{
\soc(R_{\ul n})^\vee\ar[d]^-\cong & \soc(R_{\ul m})^\vee\ar[l]^-\cong\ar[d]^-\cong\\
W_{\ul n}\otimes\KK & W_{\ul m}\otimes\KK,\ar[l]_-{\ul z^{\ul m-\ul n}}^-\cong
}
\]
and hence the first isomorphism in \eqref{25l}.

Consider the short exact sequence of inverse systems represented by
\begin{equation}\label{35}
0\to\mm_P\cdot W_\ul m\to W_\ul m\to W_\ul m\otimes\KK\to 0.
\end{equation}
Since $\dim_\KK W_\ul m<\infty$ by Lemma~\ref{22}.\eqref{22a}, the left inverse system in \eqref{35} satisfies the Mittag--Leffler condition.
Therefore, the inverse limit preserves exactness when applied to \eqref{35}.
This yields the surjection in \eqref{25l}.

\item Any $\KK$-linear section $\sigma$ of \eqref{25l} with image $H$ fits into a diagram
\[
\xymatrix{
&&& \\
H \ar[d]^-{\cong} \ar@{^{(}->}[r] & \varprojlim_{\ul n\in\NN^d}W_{\ul n}\ar[d]\ar@{->>}[r] & \varprojlim_{\ul n\in\NN^d}(W_{\ul n}\otimes\KK)\ar[d]^-{\cong} \ar[r]_-{\cong} & W_{\ul 1}\otimes\KK\ar@/_2em/[lll]_-{\sigma}^-{\cong}\\
H_{\ul m}\ar@/_1em/[rr]_-{\cong}\ar@{^{(}->}[r] & W_{\ul m}\ar@{->>}[r] & W_{\ul m}\otimes\KK\ar[ru]_-{\cong}
}
\]
and the claim follows by Nakayama lemma (see \cite[Ex.~1.5.24]{BH93}).\qedhere

\item By construction, $H\cong H_\ul 1\cong\soc(R_\ul 1)^\vee$.
Using that $-^\vee$ preserves length, this gives condition~\ref{27}.\eqref{27b} (see \cite[Lem.~1.2.19]{BH93}),
\begin{equation}\label{42}
\dim_\KK H=\dim_\KK H_\ul 1=\dim_\KK\soc R_\ul 1=\type R=r.
\end{equation}
For $\ul m\not\ge\ul 1$, $R_\ul m=0$, and hence $H_\ul m\subset W_\ul m=0$.
With \eqref{42}, condition~\ref{27}.\eqref{27a} follows.
Part~\eqref{21c} with $\ul m=\ul 1$ gives $\ideal{H_\ul1}=W_\ul 1=I_\ul 1^\perp$, and hence $\max\deg H_\ul1=\socdeg R_\ul 1=s$ by \eqref{10} and \eqref{34}.
Condition~\ref{27}.\eqref{27d} follows by \eqref{10}.
By \eqref{39} and Remark~\ref{32},
\begin{align*}
(W_\ul m\cap V^{j,s-d}_\ul m)^\perp
&=W_\ul m^\perp+(V^{j,s-d}_\ul m)^\perp\\
&=I+\ideal{z_1^{m_1},\dots,z_d^{m_d}}+\ideal{\ul x}^{\abs{\ul m}+s-d+1}+\ideal{z_j^{m_j-1}}\\
&\supset I+\ideal{z_1^{m_1},\dots,z_j^{m_j-1},\dots,z_d^{m_d}}
=W_{\ul m-\ul e_j}^\perp.
\end{align*}
Condition~\ref{27}.\eqref{27c} follows with \eqref{39}.

\end{asparaenum}
\end{proof}


\begin{lem}\label{31}
Let $H$ be in the set~\ref{25}.\eqref{25h} and $I$ as in \eqref{25a}.
\begin{enumerate}[(a)]
\item\label{31a} There is an equality $I_\ul m=W_\ul m^\perp$, for all $\ul m\in\NN^d$.
\item\label{31b} The sequence $\ul z$ maps to an $R$-sequence.
\item\label{31c} The ring $R=P/I$ is Cohen--Macaulay with $\dim R=d$.
\end{enumerate}
\end{lem}

\begin{proof}\
\begin{asparaenum}[(a)]

\item Let $j\in\set{1,\dots,d}$ and $\ul m\in\NN^d$.
By Lemma~\ref{22}.\eqref{22b}, $W_{\ul m+\ul e_j}\subset D$ is a finitely generated (graded) $P$-submodule.
Then $P/W_{\ul m+\ul e_j}^\perp$ is Artinian by \eqref{39}.
By \eqref{34} and Lemma~\ref{22}.\eqref{22b},
\[
\socdeg(P/W_{\ul m+\ul e_j}^\perp)=\max\deg W_{\ul m+\ul e_j}=\abs{\ul m+\ul e_j}+s-d,
\]
and hence 
\[
W_{\ul m+\ul e_j}^\perp\supset\ideal{\ul x}^{\abs{\ul m+\ul e_j}+s-d+1}.
\]
Using Definition~\ref{27}.\eqref{27c}, \eqref{39} and Remark~\ref{32}, it follows that
\begin{align*}
W_{\ul m}^\perp&\subset(W_{\ul m+\ul e_j}\cap V^{j,s-d}_{\ul m+\ul e_j})^\perp\\
&=W_{\ul m+\ul e_j}^\perp+(V^{j,s-d}_{\ul m+\ul e_j})^\perp\\
&=W_{\ul m+\ul e_j}^\perp+\ideal{\ul x}^{\abs{\ul m+\ul e_j}+s-d+1}+\ideal{z_j^{m_j}}\\
&=W_{\ul m+\ul e_j}^\perp+\ideal{z_j^{m_j}}.
\end{align*}
This already implies that (see \cite[Prop.~10, Claim~1]{MT18})
\[
W_{\ul m}^\perp\subset I+\ideal{z_1^{m_1},\dots,z_d^{m_d}}.
\]
The opposite inclusion holds true since $I\subset W_\ul m^\perp$ by definition and  $z_i^{m_i}\cdot H_{\ul m}=0$, for all $i\in\set{1,\dots,d}$, by Remark~\ref{29}.\eqref{29a}.

\item By dualizing surjections $\ul z^{\ul m-\ul n}\colon W_\ul m\onto W_\ul n$ for suitable $\ul m,\ul n\in\NN^d$ with $\ul n\le\ul m$, one shows that $\ul z$ maps to a weak $R$-sequence (see \cite[Prop.~10, Claim~2]{MT18}).
Since $W_{\ul 1}\ne 0$ by Definition~\ref{27}.\eqref{27a}, $I_{\ul 1}=W_{\ul 1}^\perp\ne R$ by part~\eqref{31a} with $\ul m=\ul 1$ and \eqref{39}.
Thus, $R_\ul1\ne0$ and $\ul z$ maps to an $R$-sequence.

\item By part~\eqref{31a} with $\ul m=\ul 1$, the ring $R_\ul1$ is Artinian, and hence $\dim R_\ul1=0$ by Hopkins theorem.
With \eqref{31b}, it follows that $R_{\mm_R}$, and hence $R$ is Cohen--Macaulay with $\dim R=d$ (see \cite[Ex.~2.1.27.(c)]{BH93}).\qedhere

\end{asparaenum}
\end{proof}

\begin{proof}[Proof of Theorem~\ref{25}]
This follows from \eqref{39}, Lemmas~\ref{41}, \ref{21} and \ref{31}.
\end{proof}


\begin{exa}
Let us consider the irreducible algebroid curve
\[
R=\CC\fs{t^6,t^7, t^{11},t^{13}}.
\]
Note that $R$ is not quasi-homogeneous.
We write $R\cong P/I$ where  
\begin{align*}
P&=\CC\fs{x,y,z,w},\\
I&=\ideal{w-xy,yz-x^3,xz^2-y^4,z^3-x^2y^3,y^5-x^4z}.
\end{align*}
The element $x\in P$ maps to the regular element $t^6\in R$.
It can be checked that $\type{R}=2$ and $P/(I+\ideal{x})$ has Hilbert--Samuel function $(1,2,2,1,1)$.
It follows that $R$ is not level.
Using \textsc{Singular} (see \cite{DGPS18}), we compute the socles of $R_m$ (see \eqref{9}) up to $m=3$:
\begin{align*}
\soc R_{1} &= \ideal{\ol z^2, \ol y^3}, \\
\soc R_{2} &= \ideal{\ol x\ol z^2, \ol x\ol y^3}, \\
\soc R_{3} &= \ideal{\ol x^2\ol z^2, \ol x^2\ol y^3}.
\end{align*}
They fit into the commutative diagram (see \eqref{37})
\[
\xymatrix{
	R_{1}\ar[r]^-{x} & R_{2}\ar[r]^-{x} & R_{3}\\
	\soc R_{1}\ar@{^(->}[u]\ar[r]_-\cong & \soc R_{2}\ar@{^(->}[u]\ar@{^(->}[u]\ar[r]_-\cong & \soc R_{3}.\ar@{^(->}[u]
}
\]
Using a \textsc{Singular} library by Elias (see \cite{Eli15}), we compute the limit inverse system $H$ associated to $I$ by Theorem~\ref{25} up to $m=7$:
{\tiny
\begin{align*}
H_{1} =& \ideal{Y^3,Z^2}_\KK,\\
H_{2} =& \ideal{XY^3+Y^2W,XZ^2+Y^4}_\KK,\\
H_{3} =&\ideal{X^2Y^3+XY^2W+YW^2+Z^3, X^2Z^2+XY^4+Y^3W}_\KK,\\
H_{4} =&\langle X^3Y^3+X^2Y^2W+XYW^2+XZ^3+Y^4Z+W^3,\\
& X^3Z^2+X^2Y^4+XY^3W+YZ^3+Y^2W^2\rangle_\KK,\\
H_{5} =&\langle X^4Y^3+X^3Y^2W+X^2YW^2+X^2Z^3+XY^4Z+XW^3+Y^3ZW,\\
& X^4Z^2+X^3Y^4+X^2Y^3W+XYZ^3+XY^2W^2+Z^3W+YW^3+Y^5Z\rangle_\KK,\\
H_{6} =&\langle X^5Y^3+X^3Z^3+X^4Y^2W+X^3YW^2+X^2Y^4Z+X^2W^3+XY^3ZW+Y^2ZW^2+YZ^4,\\
& X^5Z^2+X^4Y^4+X^3Y^3W+X^2YZ^3+X^2Y^2W^2+XZ^3W+XYW^3+XY^5Z+\\
& Y^4ZW+W^4\rangle_\KK,\\
H_{7} =&\langle X^6Y^3+X^5Y^2W+X^4Z^3+X^4YW^2+X^3Y^4Z+X^3W^3+X^2Y^3ZW+XY^2ZW^2+\\
& XYZ^4+Z^4W+YZW^3+Y^5Z^2,\\
& X^6Z^2+X^5Y^4+X^4Y^3W+X^3YZ^3+X^3Y^2W^2+X^2Z^3W+X^2YW^3+X^2Y^5Z+\\
& XY^4ZW+XW^4+Y^2Z^4+Y^3ZW^2\rangle_\KK.
\end{align*}}
\end{exa}

\printbibliography
\end{document}